\documentclass{amsart}
\usepackage{graphicx}
\usepackage{amssymb,amscd,amsthm,amsxtra}
\usepackage{latexsym}
\usepackage{epsfig}

\newtheorem{thm}{Theorem}[section]
\newtheorem{cor}[thm]{Corollary}
\newtheorem{lem}[thm]{Lemma}
\newtheorem{prop}[thm]{Proposition}
\theoremstyle{definition}
\newtheorem{defn}[thm]{Definition}
\theoremstyle{remark}

\numberwithin{equation}{section}

\newcommand{\R}{\mathbb R}

\newcommand{\eps}{\epsilon}

\newcommand{\p}{\partial}

\newcommand{\comment}[1]{}

\begin{document}

\title[A localization property]{A localization property at the boundary for Monge-Ampere equation}
\author{O. Savin}
\address{Department of Mathematics, Columbia University, New York, NY 10027}
\email{\tt  savin@math.columbia.edu}

\thanks{The author was partially supported by NSF grant 0701037.}

\maketitle

\section{Introduction}

In this paper we study the geometry of the sections for solutions to the Monge-Ampere equation $$\det D^2u=f, \quad \quad u:\overline \Omega \to \mathbb{R}\quad \mbox{convex},$$
which are centered at a boundary point $x_0\in \partial \Omega$. We show that under natural local assumptions on the boundary data and the domain, the sections $$S_h(x_0)=\{x \in \overline \Omega | \quad u(x)<u(x_0)+\nabla u(x_0) \cdot (x-x_0)+h\}$$ are ``equivalent" to ellipsoids centered at $x_0$, that is, for each $h>0$ there exists an ellipsoid $E_h$ such that $$cE_h \cap \overline \Omega \subset \, S_h(x_0) -x_0 \subset \, CE_h \cap \overline \Omega,$$
with $c$, $C$ constants independent of $h$.

The situation in the interior is well understood. Caffarelli showed in \cite{C1} that if $$0<\lambda \leq f \leq \Lambda \quad \text{in $\Omega$},$$ and for some $x \in \Omega$, $$S_h(x)\subset \subset \Omega,$$
then $S_h(x)$ is equivalent to an ellipsoid centered at $x$ i.e.
$$kE \subset S_h(x)-x \subset k^{-1} E$$ for some ellipsoid $E$ of volume $h^{n/2}$ and for a constant $k>0$ which depends only on $\lambda, \Lambda, n.$

This property provides compactness of sections modulo affine transformations. This is particularly useful when dealing with interior $C^{2,\alpha}$ and $W^{2,p}$ estimates of strictly convex solutions of $$\det D^2u=f$$ when $f>0$ is continuous (see \cite{C2}).

Sections at the boundary were also considered by Trudinger and Wang in \cite{TW} for solutions of $$\det D^2u = f $$ but under stronger assumptions on the boundary behavior of $u$ and $\p \Omega$, and with $f\in C^\alpha(\overline \Omega).$ They proved $C^{2,\alpha}$ estimates up to the boundary by bounding the mixed derivatives and obtained that the sections are equivalent  to balls.

\section{Statement of the main Theorem.}

Let $\Omega$ be a bounded convex set in $\R^n$. We assume throughout this note that
\begin{equation}\label{om_ass}
B_\rho(\rho e_n) \subset \, \Omega \, \subset \{x_n \geq 0\} \cap B_{\frac 1\rho},
\end{equation}
for some small $\rho>0$, that is $\Omega \subset (\R^n)^+$ and $\Omega$ contains an interior ball tangent to $\p \Omega$ at $0.$

Let $u : \overline \Omega \rightarrow \R$ be convex, continuous, satisfying
\begin{equation}\label{eq_u}
\det D^2u =f, \quad \lambda \leq f \leq \Lambda \quad \text{in $\Omega$}.
\end{equation}
We extend $u$ to be $\infty$ outside $\overline \Omega.$

By subtracting a linear function we may assume that
\begin{equation}\label{eq_u1}
\mbox{$x_{n+1}=0$ is the tangent plane to $u$ at $0$,}
\end{equation}
in the sense that $$u \geq 0, \quad u(0)=0,$$
and any hyperplane $x_{n+1}= \eps x_n$, $\eps>0$ is not a supporting hyperplane for $u$.

In this paper we investigate the geometry of the sections of $u$ at $0$ that we denote for simplicity of notation
$$S_h := \{x \in \overline \Omega : \quad u(x) < h \}.$$

We show that if the boundary data has quadratic growth near $\{x_n=0\}$ then, as $h \rightarrow 0$, $S_h$ is equivalent to a half-ellipsoid centered at 0.

Precisely, our main theorem reads as follows.

\begin{thm}\label{main} Assume that $\Omega$, $u$ satisfy \eqref{om_ass}-\eqref{eq_u1} above and for some $\mu>0$,
\begin{equation}\label{commentstar}\mu |x|^2 \leq u(x) \leq \mu^{-1} |x|^2 \quad \text{on $\p \Omega \cap \{x_n \leq \rho\}.$}\end{equation}
Then, for each $h<c(\rho)$ there exists an ellipsoid $E_h$ of volume $h^{n/2}$ such that
$$kE_h \cap \overline \Omega \, \subset \, S_h \, \subset \, k^{-1}E_h .$$

Moreover, the ellipsoid $E_h$ is obtained from the ball of radius $h^{1/2}$ by a
linear transformation $A_h^{-1}$ (sliding along the $x_n=0$ plane)
$$A_hE_h= h^{1/2}B_1$$
$$A_h(x) = x - \nu x_n, \quad \nu = (\nu_1, \nu_2, \ldots, \nu_{n-1}, 0), $$
with
$$ |\nu| \leq k^{-1} |\log h|.$$
The constant $k$ above depends on $\mu, \lambda, \Lambda, n$ and $c(\rho)$ depends also on $\rho$.
\end{thm}

Theorem \ref{main} is new even in the case when $f=1$. The ellipsoid $E_h$, or equivalently the linear map $A_h$, provides information about the behavior of the second derivatives near the origin. Heuristically, the theorem states that in $S_h$ the tangential second derivatives are bounded from above and below and the mixed second derivatives are bounded by $|\log h|$. This is interesting given that $f$ is only bounded and the boundary data and $\p \Omega$ are only $C^{1,1}$ at the origin.

\

\textbf{Remark.} Given only the boundary data $\varphi$ of $u$ on $\p \Omega$, it is not always easy to check condition \eqref{commentstar}. Here we provide some examples when \eqref{commentstar} is satisfied:

\

1) If $\varphi$ is constant and the domain $\Omega$ is included in a ball included in $\{x_n \geq 0\}.$

 \

 2) If the domain $\p \Omega$ is tangent of order 2 to $\{x_n=0\}$ and the boundary data $\varphi$ has quadratic behavior in a neighborhood of 0.

 \

3) $\varphi$, $\p \Omega \in C^3$ at the origin, and $\Omega$ is uniformly convex at the origin.

\

We obtain  compactness of sections modulo affine transformations.

\begin{cor} Under the assumptions of Theorem \ref{main}, assume that
$$\lim _{x \rightarrow 0} f(x) = f(0)$$ and
$$u(x) = P(x) + o(|x|^2) \quad \mbox {on $\p \Omega$}$$ with $P$ a quadratic polynomial. Then we can find a sequence of rescalings $$\tilde u_h(x) : = \frac 1 h u(h^{1/2} A_h^{-1} x)$$ which converges to a limiting continuous solution $\bar u_0 : \overline \Omega_0 \rightarrow \R$ with $$kB_1^+ \subset \Omega_0 \subset k^{-1}B_1^+$$ such that
$$\det D^2 \bar u_0 = f(0)$$ and
\begin{align*}
& \bar u_0 = P \quad \text{on $\overline \Omega_0 \cap \{x_n=0\}$,}\\
& \bar u_0 =1 \quad \text{on $\p \overline \Omega_0 \cap \{x_n>0\}.$}
\end{align*}

\end{cor}

In a future work we intend to use the results above and obtain $C^{2, \alpha}$ and $W^{2,p}$ boundary estimates under appropriate conditions on the domain and boundary data.

\section{Preliminaries}

Next proposition was proved by Trudinger and Wang in \cite{TW}. Since our setting is slightly different we provide its proof.

\begin{prop}\label{TW} Under the assumptions of Theorem \ref{main}, for all $h \leq c(\rho),$ there exists a linear transformation (sliding along $x_n=0$) $$A_h(x) = x - \nu x_n,$$ with$$ \nu_n=0, \quad |\nu|\leq C(\rho) h^{-\frac{n}{2(n+1)}}$$ such that the rescaled function
$$\tilde u(A_h x) = u(x),$$ satisfies in $$\tilde S_h := A_h S_h = \{\tilde u<h\}$$ the following:
\begin{enumerate}
\item the center of mass of $\tilde S_h$ lies on the $x_n$-axis;
\item $$k_0 h^{n/2} \leq |\tilde S_h| = |S_h| \leq k_0^{-1} h^{n/2};$$
\item the part of $\p \tilde S_h$ where $\{\tilde u <h\}$ is a graph, denoted by $$\tilde G_h = \p \tilde S_h \cap \{\tilde u <h\} = \{(x', g_h(x'))\}$$ that satisfies $$g_h \leq C(\rho)|x'|^2$$ and $$\frac \mu 2 |x'|^2 \leq \tilde u \leq 2\mu^{-1} |x'^2| \quad \text{on $\tilde G_h$}.$$
\end{enumerate}

The constant $k_0$ above depends on $\mu, \lambda, \Lambda, n$ and the constants $C(\rho), c(\rho)$ depend also on $\rho$.
\end{prop}

\

 In this section we denote by $c$, $C$ positive constants that depend on $n$, $\mu$, $\lambda$, $\Lambda$. For simplicity of notation, their values may change from line to line whenever there is no possibility of confusion. Constants that depend also on $\rho$ are denote by $c(\rho)$, $C(\rho)$.

\begin{proof}
The function
$$v: = \mu |x'|^2 + \frac{\Lambda}{\mu^{n-1}} x_n^2 -C(\rho) x_n$$ is a lower barrier for $u$ in $\Omega \cap \{x_n \leq \rho\}$ if $C(\rho)$ is chosen large.

Indeed, then $$v \leq u \quad \text{on $\p \Omega \cap \{x_n \leq \rho\}$},$$

$$v \leq 0 \leq u \quad \text {on $\Omega \cap \{x_n=\rho\}$},$$
and
$$\det D^2 v > \Lambda.$$
In conclusion, $$v \leq u \quad \text{in $\Omega \cap \{x_n \leq \rho\}$},$$ hence
\begin{equation}\label{star} S_h \cap \{x_n \leq \rho\} \subset \{v <h\} \subset \{x_n > c(\rho)(\mu |x'|^2- h)\}.
\end{equation}

Let $x^*_h$ be the center of mass of $S_h.$ We claim that
\begin{equation}\label{2star}x^*_h \cdot e_n \geq c_0(\rho) h^{\alpha}, \quad \alpha= \frac{n}{n+1},\end{equation}
 for some small $c_0(\rho)>0$.

Otherwise, from \eqref{star} and John's lemma we obtain
$$S_h \subset \{x_n \leq C(n) c_0 h^{\alpha} \leq h^\alpha\} \cap \{|x'| \leq C_1 h^{\alpha/2}\},$$ for some large $C_1=C_1(\rho)$. Then the function
$$w= \eps x_n + \frac{h}{2} \left(\frac{|x'|}{C_1h^{\alpha/2}}\right)^2 + \Lambda C_1 ^{2(n-1)} h \left(\frac{x_n}{h^\alpha}\right)^2$$
is a lower barrier for $u$ in $S_h$ if $c_0$ is sufficiently small.

Indeed,
$$w \leq \frac h 4 + \frac h 2 + \Lambda C_1 ^{2(n-1)}(C(n)c_0 )^2 h < h \quad \text{in $S_h,$}$$
and for all small $h$,
$$w \leq \eps x_n + \frac{h^{1-\alpha}}{C_1 ^2} |x'|^2 + C(\rho)hc_0\frac{x_n}{h^\alpha} \leq \mu |x'|^2 \leq u \quad \text{on $\p \Omega$,}$$
and
$$\det D^2 w = 2\Lambda.$$

Hence $$w \leq u \quad \text{in $S_h$,}$$ and we contradict that 0 is the tangent plane at 0. Thus claim \eqref{2star} is proved.

Now, define $$A_h x = x - \nu x_n, \quad \nu = \frac{x^{*'}_h}{x_h^* \cdot e_n},$$
and
$$\tilde u(A_h x) = u(x).$$

The center of mass of $\tilde S_h=A_hS_h$ is $$\tilde x ^*_h=A_hx^*_h$$ and lies on the $x_n$-axis from the definition of $A_h$.
Moreover, since $x^*_h \in S_h$, we see from \eqref{star}-\eqref{2star} that
$$|\nu| \leq C(\rho) \frac{(x_h^*\cdot e_n)^{1/2}}{(x_h^*\cdot e_n)} \leq C(\rho) h^{-\alpha/2},$$
and this proves (i).

If we restrict the map $A_h$ on the set on $\p \Omega$ where $\{u < h\}$, i.e. on
$$\p S_h \cap \p \Omega \subset \{x_n \leq \frac{|x'|^2}{\rho}\} \cap \{|x'| < Ch^{1/2}\}$$
we have
$$|A_h x - x| = |\nu| x_n \leq C(\rho) h^{-\alpha/2} |x'|^2  \leq C(\rho) h^{\frac{1-\alpha}{2}} |x'|,$$
and part (iii) easily follows.

Next we prove (ii). From John's lemma, we know that after relabeling the $x'$ coordinates if necessary,
\begin{equation}\label{3star} D_h B_1 \subset \tilde S_h - \tilde x^*_h \subset C(n) D_h B_1\end{equation}where
\[
 D_h =
 \begin{pmatrix}
  d_1 & 0 & \cdots & 0 \\
  0 & d_{2} & \cdots & 0 \\
  \vdots  & \vdots  & \ddots & \vdots  \\
  0 & 0 & \cdots & d_{n}
 \end{pmatrix}.
\]

Since $$\tilde u \leq 2 \mu^{-1}|x'|^2 \quad \text{on $\tilde G_h = \{(x', g_h(x'))\}$},$$ we see that the domain of definition of $g_h$ contains a ball of radius $(\mu h/2)^{1/2}$. This implies that
$$d_i \geq c_1 h^{1/2}, \quad \quad i=1,\cdots, n-1,$$ for some $c_1$ depending only on $n$ and $\mu.$ Also from \eqref{2star} we see that $$\tilde x^*_h \cdot e_n =x^*_h \cdot e_n \ge c_0(\rho) h^\alpha$$ which gives $$d_n \ge c(n) \tilde x^*_h \cdot e_n \ge c(\rho) h^\alpha.$$

We claim that for all small $h$,
$$\prod_{i=1}^n d_i \geq k_0 h^{n/2},$$
with $k_0$ small depending only on $\mu, n, \Lambda,$ which gives the left inequality in (ii).

To this aim we consider the barrier,
$$w= \eps x_n + \sum_{i=1}^n ch\left(\frac{x_i}{d_i}\right)^2.$$ We choose $c$ sufficiently small depending on $\mu, n, \Lambda$ so that for all $h<c(\rho)$,
$$w \leq h \quad \text{on $\p \tilde S_h$,}$$
and on the part of the boundary $\tilde G_h$, we have $w \le \tilde u$ since
\begin{align*}w & \leq \eps x_n+\frac{c}{c_1^2}|x'|^2 + c h \left(\frac {x_n}{d_n}\right)^2 \\
 & \leq \frac \mu 4 |x'|^2 + c h C(n) \frac{x_n}{d_n} \\
 & \leq \frac \mu 4 |x'|^2 + c h^{1-\alpha}C(\rho)|x'|^2\\
 &\le \frac \mu 2 |x'|^2.
 \end{align*}
Moreover, if our claim does not hold, then
$$\det D^2 w = (2c h )^n (\prod d_i)^{-2n} > \Lambda,$$ thus $w \le \tilde u$ in $\tilde S_h$. By definition, $\tilde u$ is obtained from $u$ by a sliding along $x_n=0$, hence $0$ is still the tangent plane of $\tilde u$ at $0$. We reach again a contradiction since $\tilde u \ge w\ge \eps x_n$ and the claim is proved.

Finally we show that $$|\tilde S_h| \leq Ch^{n/2}$$ for some $C$ depending only on $\lambda, n.$ Indeed, if $$v=h \quad \text{on $\p \tilde S_h$},$$ and $$\det D^2v= \lambda$$ then
$$v \geq u \geq 0 \quad \text{in $\tilde S_h$.}$$ Since
$$h \geq h-\min_{\tilde S_h} v \geq c(n,\lambda) |\tilde S_h|^{2/n}$$ we obtain the desired conclusion.

\end{proof}

\

In the proof above we showed that for all $h \leq c(\rho),$ the entries of the diagonal matrix $D_h$ from \eqref{3star} satisfy
$$d_i \geq c h^{1/2}, \quad i=1,\ldots n-1$$

$$d_n \geq c(\rho)h^{\alpha}, \quad \alpha= \frac{n}{n+1}$$

$$c h^{n/2} \leq \prod d_i \leq Ch^{n/2}.$$

The main step in the proof of Theorem \ref{main} is the following lemma that will be proved in the remaining sections.

\begin{lem}\label{l1}
There exist constants $c$, $c(\rho)$ such that
\begin{equation}\label{dn}d_n \geq ch^{1/2},\end{equation}
for all $h \le c(\rho)$.
\end{lem}

Using Lemma \ref{l1} we can easily finish the proof of our theorem.

\

{\it Proof of Theorem \ref{main}.} Since all $d_i$ are bounded below by $c h^{1/2}$ and their product is bounded above by $Ch^{n/2}$ we see that
$$C h^{1/2} \geq d_i \geq ch^{1/2} \quad \quad i=1,\cdots,n$$ for all $h\le c(\rho)$. Using \eqref{3star} we obtain $$\tilde S_h \subset Ch^{1/2}B_1.$$ Moreover, since $$\tilde x^*_h \cdot e_n \ge d_n \ge c h^{1/2}, \quad \quad  (\tilde x^*_h)'=0,$$ and the part $\tilde G_h$ of the boundary $\p \tilde S_h$ contains the graph of $\tilde g_h$ above $|x'| \le c h^{1/2}$, we find that $$ch^{1/2}B_1 \cap \tilde \Omega \subset \tilde S_h,$$ with $\tilde \Omega=A_h \Omega$, $\tilde S_h=A_h S_h$. In conclusion
$$ch^{1/2}B_1 \cap \tilde \Omega \subset A_h S_h \subset Ch^{1/2}B_1.$$
We define the ellipsoid $E_h$ as
$$E_h:=A_h^{-1}(h^{1/2}B_1),$$
hence $$c E_h \cap \overline \Omega \subset S_h \subset C E_h.$$
Comparing the sections at levels $h$ and $h/2$ we find
$$cE_{h/2} \cap \overline \Omega \subset C E_h$$
and we easily obtain the inclusion
$$ A_hA_{h/2}^{-1} B_1 \subset C B_1.$$
If we denote $$A_hx=x-\nu_h x_n$$ then the inclusion above implies $$|\nu_h-\nu_{h/2}| \le C,$$ which gives the desired bound $$|\nu_h| \le C|\log h|$$ for all small $h$.

\qed

\

We introduce a new quantity $b(h)$ which is proportional to $d_n h^{-1/2}$ and which is appropriate when dealing with affine transformations.

\

\textbf{Notation.} Given a convex function $u$ we define $$b_u(h) = h^{-1/2} \sup_{S_h} x_n.$$ Whenever there is no possibility of confusion we drop the subindex $u$ and use the notation $b(h)$.

\

Below we list some basic properties of $b(h)$.

\

1) If $h_1 \le h_2$ then $$\left(\frac{h_1}{h_2}\right)^\frac 12 \le \frac{b(h_1)}{b(h_2)} \le \left(\frac{h_2}{h_1}\right)^\frac 12.$$

2) A rescaling $$ \tilde u (Ax) =u(X)$$ given by a linear transformation $A$ which leaves the $x_n$ coordinate invariant does not change the value of $b,$ i.e $$b_{\tilde u}(h)=b_u(h).$$

3) If $A$ is a linear transformation which leaves the plane $\{x_n=0\}$ invariant the values of $b$ get multiplied by a constant. However the quotients $b(h_1)/b(h_2)$ do not change values i.e $$\frac{b_{\tilde u}(h_1)}{b_{\tilde u}(h_2)}=\frac{b_u(h_1)}{b_u(h_2)}.$$

4) If we multiply $u$ by a constant, i.e.
$$\tilde u(x) = \beta u(x)$$
then $$ b_{\tilde u}(\beta h)= \beta^{-1/2}b_u(h),$$ and $$ \frac{b_{\tilde u}(\beta h_1)}{b_{\tilde u}(\beta h_2)}=\frac{b_u(h_1)}{b_u(h_2)}.$$

\

From \eqref{3star} and property 2 above,
$$c(n)d_n \leq b(h)h^{1/2} \leq C(n)d_n,$$
hence Lemma \ref{l1} will follow if we show that $b(h)$ is bounded below. We achieve this by proving the following lemma.

 \begin{lem}\label{l2} There exist $c_0$, $c(\rho)$ such that if $h \le c(\rho)$ and $b(h) \le c_0$ then \begin{equation}\label{quo}\frac{b(t h)}{b(h)} >2,
 \end{equation} for some $t \in [c_0, 1]$.
 \end{lem}
  This lemma states that if the value of $b(h)$ on a certain section is less than a critical value $c_0$, then we can find a lower section at height still comparable to $h$ where the value of $b$ doubled. Clearly Lemma \ref{l2} and property 1 above imply that $b(h)$ remains bounded for all $h$ small enough.

The quotient in \eqref{quo} is the same for $\tilde u$ which is defined in Proposition \ref{TW}. We normalize the domain $\tilde S_h$ and $\tilde u$ by considering the rescaling
$$v(x)= \frac{1}{h} \tilde u(h^{1/2}Ax)$$
where $A$ is a multiple of $D_h$ (see \eqref{3star}), $A=\gamma D_h$ such that $$\det A=1.$$ Then $$ch^{-1/2} \le \gamma \le C h^{-1/2},$$ and the diagonal entries of $A$ satisfy $$a_i \ge c , \quad \quad i=1,2,\cdots, n-1,$$ $$ c b_u(h) \le a_n \le Cb_u(h).$$

The function $v$ satisfies
$$\lambda \leq \det D^2v \leq \Lambda,$$

$$v \geq 0, \quad  v(0)=0,$$
is continuous and it is defined in $\bar \Omega_v$ with

$$\Omega_v:= \{v<1\} = h^{-1/2}A^{-1}\tilde S_h.$$ Then

$$x^* + cB_1 \subset \Omega_v \subset C B_1^+,$$ for some $x^*$, and

$$ct^{n/2} \leq |S_t(v)| \leq Ct^{n/2}, \quad \forall t\leq 1,$$
where $S_t(v)$ denotes the section of $v$. Since $$\tilde u=h \quad \mbox{in} \quad \p \tilde S_h \cap \{x_n \ge C(\rho) h\},$$ then
$$v=1 \quad \text{on $\p \Omega_v \cap \{x_n \geq \sigma\}, \quad \sigma:=C(\rho)h^{1-\alpha}$}.$$
Also, from Proposition \ref{TW} on the part $G$ of the boundary of $\p \Omega_v$ where $\{v<1\}$ we have
 \begin{equation}\label{ai}
 \frac 1 2 \mu \sum_{i=1}^{n-1} a_i^2 x_i^2 \leq v \leq 2 \mu^{-1} \sum_{i=1}^{n-1} a_i^2 x_i^2.
 \end{equation}

In order to prove Lemma \ref{l2} we need to show that if $\sigma$, $a_n$ are sufficiently small depending on $n, \mu, \lambda, \Lambda$ then the function $v$ above satisfies
\begin{equation}\label{v}
b_v (t) \geq 2b_v(1)
\end{equation} for some $1 > t \geq c_0.$

Since $\alpha<1$, the smallness condition on $\sigma$ is satisfied by taking $h<c(\rho)$ sufficiently small. Also $a_n$ being small is equivalent to one of the $a_i$, $1 \le i \le n-1$ being large since their product is 1 and $a_i$ are bounded below.

In the next sections we prove property \eqref{v} above by compactness, by letting $\sigma \to 0$, $a_i \to \infty$ for some $i$. First we consider the 2D case and in the last section the general case.

\section{The 2 dimensional case.}

In order to fix ideas, we consider first the 2 dimensional case.

We study the following class of solutions to the Monge-Ampere equation. Fix $\mu>0$ small, $\lambda, \Lambda.$ We denote by $\mathcal{D}_\sigma$ the set of convex, continuous functions
$$u : \overline \Omega \rightarrow \R$$ such that
\begin{align}
\label{1} & \lambda \leq \det D^2u \leq \Lambda;\\
\label{2} & 0 \in \p\Omega, \quad B_{\mu}(x_0) \subset \Omega \subset B_{1/\mu}^+ \quad \text{for some $x_0$;}\\
\label{3} & \mu h^{n/2} \leq |S_h| \leq \mu^{-1} h^{n/2};
\end{align}
\begin{equation}
u =1 \quad \text{on $\p \Omega \setminus G$}, \quad \quad 0 \le u \le 1 \quad \text{on $G$,} \quad u(0)=0,
\end{equation} with $G$ a closed subset of $\p \Omega$ included in $B_\sigma,$
$$G \subset \ \p \Omega \cap B_\sigma.$$

\begin{prop}\label{2dprop} Assume $n=2$. For any $M>0$ there exists $c_0$ small depending on $M, \mu, \lambda, \Lambda,$ such that if $u \in \mathcal{D}_\sigma$ and $\sigma \leq c_0,$ then
$$b(h) : = (\sup_{S_h} x) h^{-1/2} > M$$ for some $h \geq c_0.$
\end{prop}

 Property \eqref{v} easily follows from the proposition above. Indeed, by choosing $$M= 2 \mu^{-1} >2b(1)$$ we prove the existence of a section $h \geq c_0$ such that $$b(h) \geq 2 b(1).$$ Also, the function $v$ of the previous section satisfies $v \in \mathcal{D}_{c_0}$ (after renaming the constant $\mu$) provided that $\sigma$ is sufficiently small and $a_1$ sufficiently large.

\

We prove Proposition \ref{2dprop} by compactness. First we discuss briefly the compactness of bounded solutions to Monge-Ampere equation. For this we need to introduce solutions with possibly discontinuous boundary data.

Let $u:\Omega \to \R$ be a convex function with $\Omega \subset \R^n$ bounded and convex. We denote by
$$\Gamma_u:=\{(x,x_{n+1}) \in \Omega \times \R | \quad x_{n+1} \ge u(x)\}$$ the upper graph of $u$.

\begin{defn}\label{bv1}
We define the values of $u$ on $\p \Omega$ to be equal to $\varphi$ i.e
$$u|_{\p \Omega}=\varphi,$$ if the upper graph of $\varphi:\p \Omega \to \R \cup\{\infty\}$
$$\Phi:=\{(x,x_{n+1}) \in \p \Omega \times \R | \quad x_{n+1} \ge \varphi(x)\}$$
is given by the closure of $\Gamma_u$ restricted to $\p \Omega \times \R$,
$$\Phi:=\overline \Gamma_u \cap (\p \Omega \times \R ).$$
\end{defn}

From the definition we see that $\varphi$ is always lower semicontinuous. The following comparison principle holds:
if $w:\overline{\Omega} \to \R$ is continuous and $$\det D^2w \ge \Lambda \ge \det D^2u, \quad \quad w|_{\p \Omega} \le u|_{\p \Omega},$$ then $$w \le u \quad \mbox{in $\Omega$}.$$
Indeed, from the continuity of $w$ we see that for any $\varepsilon>0$, there exists a small neighborhood of $\p \Omega$ where $w-\varepsilon<u$. This inequality holds in the interior from the standard comparison principle, hence $w \le u$ in $\Omega$.

Since the convex functions are defined on different domains we use the following notion of convergence.

\begin{defn}We say that the convex functions $u_m : \Omega_m \rightarrow \R$ converge to $u : \Omega \rightarrow \R$ if the upper graphs converge $$\overline \Gamma_{u_m} \to \overline \Gamma_u \quad \mbox{in the Hausdorff distance.}$$

Similarly, we say that the lower semicontinuous functions $\varphi_m:\p \Omega_m \to \R$ converge to $\varphi : \p \Omega \rightarrow \R$ if the upper graphs converge $$\Phi_m \to \Phi \quad \mbox{in the Hausdorff distance.}$$
\end{defn}

Clearly  if $u_m$ converges to $u$, then $u_m$ converges uniformly to $u$ in any compact set of $\Omega,$ and $\Omega_m \to \Omega$ in the Hausdorff distance.

\

{\it Remark:} When we restrict the Hausdorff distance to the nonempty closed sets of a compact set we obtain a compact metric space. Thus, if $\Omega_m$, $u_m$ are uniformly bounded then we can always extract a subsequence $m_k$ such that  $u_{m_k} \to u$ and $u_{m_k}|_{\p \Omega_{m_k}} \to \varphi$.

Next lemma gives the relation between the boundary data of the limit $u$ and $\varphi$.

\begin{lem}\label{last} Let $u_m : \Omega_m \rightarrow \R$ be convex functions, uniformly bounded, such that $$\lambda \leq \det D^2 u_m \leq \Lambda$$ and $$ u_m \to u, \quad u_m |_{\p \Omega_m} \to \varphi.$$

Then $$\lambda \le \det D^2 u \le \Lambda,$$ and the boundary data of $u$ is given by $\varphi^*$ the convex envelope of $\varphi$ on $\p \Omega.$
\end{lem}

\begin{proof}Clearly $\Phi \subset \overline{\Gamma}_u$, hence $\Phi^* \subset \overline{\Gamma}_u$. It remains to show that the convex set $K$ generated by $\Phi$ contains $\overline{\Gamma}_u \cap (\p \Omega \times \R).$

Indeed consider a hyperplane $$x_{n+1} = l(x)$$
which lies strictly below $K.$ Then, for all large $m$ $$\{u_m - l \le  0\} \subset  \Omega_m,$$ and by Alexandrov estimate we have that
$$u_m - l \geq -Cd_m^{1/n}$$ where $d_m(x)$ represents the distance from $x$ to $\p \Omega_m.$ By taking $m \rightarrow \infty$ we see that
$$u - l \geq -C d^{1/n}$$ thus no point on $\p \Omega$ below $l$ belongs to $\overline \Gamma_u.$

\end{proof}

In view of the lemma above we introduce the following notation.
\begin{defn}\label{bv2}
Let $\varphi:\p \Omega \to \R$ be a lower semicontinuous function. When we write that a convex function $u$ satisfies
$$u=\varphi \quad \mbox{on $\p \Omega$}$$ we understand $$u|_{\p \Omega}=\varphi^*$$ where $\varphi^*$ is the convex envelope of $\varphi$ on $\p \Omega$.
\end{defn}

Whenever $\varphi^*$ and $\varphi$ do not coincide we can think of the graph of $u$ as having a vertical part on $\p \Omega$ between $\varphi^*$ and $\varphi$.

It follows easily from the definition above that the boundary values of $u$ when we restrict to the domain $$\Omega_h:=\{u<h\}$$ are given by
$$\varphi_h=\varphi \quad \mbox{on}\quad \p \Omega \cap \{\varphi \le h\} \subset \p \Omega_h$$ and $\varphi_h=h$ on the remaining part of $\p \Omega_h$.

The comparison principle still holds. Precisely, if $w:\overline{\Omega} \to \R$ is continuous and $$\det D^2w \ge \Lambda \ge \det D^2u, \quad \quad w|_{\p \Omega} \le \varphi,$$ then $$w \le u \quad \mbox{in $\Omega$}.$$

The advantage of introducing the notation of Definition \ref{bv2} is that the boundary data is preserved under limits.

\begin{prop}[Compactness]\label{comp}
Assume $$\lambda \le \det D^2u_m \le \Lambda, \quad u_m=\varphi_m \quad \mbox{on $\p \Omega_m$},$$ and $\Omega_m$, $\varphi_m$ uniformly bounded.

Then there exists a subsequence $m_k$ such that $$u_{m_k} \to u, \quad \varphi_{m_k} \to \varphi$$ with
$$\lambda \le \det D^2u \le \Lambda, \quad u=\varphi \quad \mbox{on $\p \Omega$}.$$

\end{prop}

Indeed, we see that we can also choose $m_k$ such that $\varphi^*_{m_k} \to \psi$. Since $\varphi_{m_k} \to \varphi$ we obtain $$\varphi \ge \psi \ge \varphi^*,$$ and the conclusion follows from Lemma \ref{last}.

\

Now we are ready to prove Proposition \ref{2dprop}.

\textit{Proof of Proposition \ref{2dprop}.} If $c_0$ does not exist we can find a sequence of functions $u_m \in \mathcal{D}_{1/m}$ such that $$b_{u_m}(h) \leq M, \quad \forall h \geq \frac 1 m.$$

By Proposition \ref{comp} there is a subsequence which converges to a limiting function $u$ satisfying \eqref{1}-\eqref{2}-\eqref{3} and (see Definition \ref{bv2}) $u=\varphi$ on $\p \Omega$ with
\begin{equation}\label{4proof}
\varphi=1 \quad \mbox{on $\p \Omega \setminus\{0\}$}, \quad \quad \varphi(0)=0,
\end{equation}
and moreover $u$ has an obstacle by below in $\Omega$ 
\begin{equation}\label{6proof} u \geq \frac{1}{M^2} x_2^2.\end{equation}

We consider the barrier $$w:= \delta (|x_1| + \frac 1 2 x_1^2) + \frac{\Lambda}{\delta} x_2^2 - N x_2$$ with $\delta$ small depending on $\mu$, and $N$ large so that
$$\frac \Lambda \delta x_2^2 - Nx_2\le 0 \quad \mbox{in} \quad B_{1/\mu}^+.$$
Then
$$w \leq \varphi \quad \text{on $\p \Omega$},$$
and
$$\det D^2 w > \Lambda.$$

Hence $$w \leq u \quad \text{in $\Omega$}$$ which gives
$$u \geq \delta |x_1| - Nx_2.$$

Next we construct another explicit subsolution $v$ such that whenever $v$ is above the two obstacles
$$\delta |x_1| - Nx_2, \quad \frac{1}{M^2} x_2^2,$$ we have
$$\det D^2v > \Lambda \quad \text{and} \quad v\le 1.$$ Then we can conclude that $$u \geq v,$$ and we show that this contradicts the lower bound on $|S_h|$.

We look for a function of the form $$v:= r f(\theta) + \frac{1}{2M^2} x_2^2,$$ where $r,\theta$ represent the polar coordinates in the $x_1,x_2$ plane.

The domain of definition of $v$ is the angle $$K:= \{\theta_0 \leq \theta \leq \pi - \theta_0\}$$ with $\theta_0$ small so that
$$\frac{1}{2M^2}x_2^2 \le \frac 1 2 (\delta |x_1| - Nx_2) \quad \text{on $\p K \cap B_\mu$}.$$

In the set $$\{v \ge \frac{1}{M^2}x_2^2\}$$ i.e. where
$$\frac{1}{r} \ge \frac{\sin^2 \theta}{2M^2f}$$ we have
\begin{equation}\label{stella}\det D^2v = \frac 1 r (f''+f) \frac{\sin^2\theta}{M^2} \ge  \frac 1 f (f''+f) \frac{\sin^4\theta_0}{2 M^4} .\end{equation}

We let $$f(\theta)= \sigma e^{C_0|\frac \pi 2 - \theta|},$$ where $C_0$ is large depending on $\theta_0, M, \Lambda$ so that (see \eqref{stella})
$$\det D^2 v>\Lambda$$ in the set where $$\{v \ge \frac{1}{M^2}x_2^2\}.$$

On the other hand we can choose $\sigma$ small so that $$v \le \delta |x_1| - Nx_2 \quad \text{on $\p K \cap B_\mu$}$$ and $$v \le 1 \quad \text{on the set $\{v \ge \frac{1}{M^2}x_2^2\}$}.$$

In conclusion $$u \geq v \geq \eps x_2,$$ hence $$u \geq \max \{\eps x_2,  \delta |x_1| - Nx_2\}.$$ This implies $$|S_h| \leq Ch^2$$ for all small $h$ and we contradict that $$|S_h| \geq \mu h, \quad \forall h \in [0,1].$$ \qed

\section{The higher dimensional case}

In higher dimensions it is more difficult to construct an explicit barrier as in Proposition \ref{2dprop} in the case when in \eqref{ai} only one  $a_i$ is large and the others are bounded. We prove our result by induction depending on the number of large eigenvalues $a_i$.

Fix $\mu$ small and $\lambda, \Lambda.$ For each increasing sequence $$\alpha_1\leq \alpha_2\leq \ldots \leq \alpha_{n-1}$$
with $$\alpha_1 \ge \mu,$$
we consider the family of solutions
$$\mathcal{D}_\sigma^\mu(\alpha_1, \alpha_2, \ldots, \alpha_{n-1})$$
of convex, continuous functions $u : \overline \Omega \rightarrow \R$ that satisfy
\begin{equation}\label{HD1} \lambda \leq \det D^2u \leq \Lambda \quad \text{in $\Omega$,} \quad \text{$u\geq 0$ in $\overline \Omega$};
\end{equation}
\begin{equation}\label{HD2}0 \in \p\Omega, \quad B_{\mu}(x_0) \subset \Omega \subset B_{1/\mu}^+ \quad \text{for some $x_0$;}
\end{equation}\begin{equation}\label{HD3} \mu h^{n/2} \leq |S_h| \leq \mu^{-1} h^{n/2};
\end{equation}
\begin{equation}\label{HD4}
u =1 \quad \text{on $\p \Omega \setminus G$}; \end{equation}
and
\begin{equation}\label{HD5} \mu \sum_{1}^{n-1} \alpha_i^2 x_i^2 \leq u \leq  \mu^{-1} \sum_{1}^{n-1} \alpha_i^2 x_i^2 \quad \quad \text{on $G$,}
\end{equation}
where $G$ is a closed subset of $\p \Omega$ which is a graph in the $e_n$ direction and is included in boundary in $\{x_n \leq \sigma\}$.

For convenience we would like to add the limiting solutions when $\alpha_{k+1} \rightarrow \infty$ and $\sigma \rightarrow 0.$ We denote by $$\mathcal{D}_0^\mu(\alpha_1,\ldots, \alpha_k, \infty, \infty, \ldots, \infty)$$ the class of functions $u:\Omega \to \R$ that satisfy properties \eqref{HD1}-\eqref{HD2}-\eqref{HD3} and (see Definition \ref{bv2}) $u=\varphi$ on $\p \Omega$ with
\begin{equation}\label{4'} \varphi=1 \quad \text{on $\p \Omega \setminus G$};\end{equation}
\begin{equation}\label{4''}\mu \sum_{1}^{k} \alpha_i^2 x_i^2 \leq \varphi \leq \min\{1, \, , \mu^{-1} \sum_{1}^{k} \alpha_i^2 x_i^2\} \quad \text{on $G$,} \end{equation}
where $G$ is a closed set $$G \subset \p \Omega \cap \{x_i=0, \quad i>k\},$$ 
and if we restrict to the space generated by the first $k$ coordinates then
$$
\{\, \mu^{-1}   \sum_{1}^k \alpha_i^2 x_i^2 \le 1  \, \} \subset G \subset  \{ \, \mu   \sum_{1}^k \alpha_i^2 x_i^2 \le 1   \, \}.$$

We extend the definition of $\mathcal{D}_\sigma^\mu(\alpha_1, \alpha_2, \ldots, \alpha_{n-1})$ to include also the pairs with $$\mu \le \alpha_1 \leq \ldots \le \alpha_k < \infty, \quad \alpha_{k+1}= \cdots=\alpha_{n-1}=\infty$$
for which $\sigma =0$ i.e. $\mathcal{D}_0^\mu(\alpha_1, \alpha_2, \ldots, \alpha_{k}, \infty, \ldots, \infty).$

Proposition \ref{comp} implies that if $$u_m \in D_{\sigma_m}^\mu(a^m_1,\ldots, a^m_{n-1})$$ is a sequence with $$\sigma_m \to 0 \quad \mbox{ and } \quad a^m_{k+1} \to \infty$$ for some fixed $0 \le k \le n-2$, then we can extract a convergent subsequence to a function $u$ with
$$u \in D_0^\mu(a_1,..,a_l,\infty,..,\infty) \quad ,$$
for some $l \le k$ and $a_1 \le \ldots \le a_l.$

\begin{prop}\label{HDprop} For any $M>0$ and $1 \leq k \leq n-1$ there exists $C_k$ depending on $M, \mu, \lambda, \Lambda,n, k$ such that if $u \in \mathcal{D}_\sigma^\mu(\alpha_1, \alpha_2, \ldots, \alpha_{n-1})$ with $$\alpha_k \geq C_k, \quad \sigma \leq C_k^{-1}$$ then $$b(h)= (\sup_{S_h} x_n) h^{-1/2} \geq M$$ for some $h$ with $ C_k^{-1} \leq h \leq 1 .$\end{prop}

As we remarked in the previous section, property \eqref{v} and therefore Lemma \ref{l2} follow from Proposition \ref{HDprop} by taking $k=n-1$ and $M=2\mu^{-1}$.

\

We prove the proposition by induction on $k$.

\begin{lem}\label{base} Proposition \ref{HDprop} holds for $k=1$.\end{lem}
\begin{proof} By compactness we need to show that there does not exist $u \in \mathcal{D}_0^\mu(\infty, \ldots, \infty)$ with $b(h) \leq M$ for all $h$.

The proof is almost identical to the 2 dimensional case. One can see as before that $$u \geq \max \{\delta|x'| - Nx_n, \frac{1}{M^2} x_n^2\}$$
and then construct a barrier of the form
$$v= r f(\theta) + \frac{1}{2M^2} x_n^2, \quad \theta_0 \leq \theta \leq \frac \pi 2$$
where $r=|x|$ and $\theta$ represents the angle in $[0, \pi/2]$ between the ray passing through $x$ and the $\{x_n=0\}$ plane.

Now,
$$\det D^2 v= \frac{f''+f}{r} \left(\frac{f\cos \theta - f'\sin \theta}{r\cos \theta}\right)^{n-2} \frac{\sin^2 \theta}{M^2}.$$

We have $$\frac f r > \frac{\sin^2\theta}{2 M^2} \quad \text{on the set $\{v > \frac{1}{M^2} x_n^2\}$}$$ and we choose a function of the form $$f(\theta):= \nu e^{C_0 (\frac \pi 2 -\theta )}$$
which is decreasing in $\theta.$

Then $$\det D^2 v > \frac{f''+f}{f} \left(\frac{\sin^2 \theta_0}{2M^2}\right)^{n-1} > \Lambda$$
if $C_0$ is chosen large.

We obtain as before that $$u \geq \max \{\delta |x'| - Nx_n, \eps x_n\}$$
which gives $$|S_h| \leq Ch^n$$ and we reach a contradiction.

\end{proof}

Now we prove Proposition \ref{HDprop} by induction on $k$.

\

\textit{Proof of Proposition \ref{HDprop}.} In this proof we denote by $c$, $C$ positive constants that depend on $M, \mu, \lambda, \Lambda, n$ and $k$.

We assume that the statement holds for $k$ and we prove it for $k+1.$

 It suffices to show the existence of $C_{k+1}$ only in the case when $\alpha_k < C_k,$ otherwise we use the induction hypothesis.

If no $C_{k+1}$ exists then we can find a limiting solution $$u \in \mathcal{D}_0^{\tilde \mu}(1,1,\ldots, 1, \infty, \ldots, \infty)$$ with \begin{equation}\label{HDstar}b(h) < M h^{1/2}, \quad \forall h>0\end{equation} where $\tilde \mu $ depends on $\mu$ and $C_k.$

We show that such a function $u$ does not exist.

Denote $$x= (y,z,x_n), \quad y=(x_1, \ldots, x_k) \in \R^k, \quad z=(x_{k+1}, \ldots, x_{n-1}) \in \R^{n-1-k}.$$

On the $\p \Omega$ plane we have
$$\varphi \geq w:=\delta |x'|^2 + \delta |z|+ \frac{\Lambda}{\delta^{n-1}} x_n^2 - N x_n$$ for some small $\delta$ depending on $\tilde \mu$, 
and $N$ large so that $$ \frac{\Lambda}{\delta^{n-1}} x_n^2 - N x_n \le 0 \quad \mbox{on} \quad B_{1/\tilde \mu}^+.$$ 
Since $$\det D^2 w > \Lambda,$$ we obtain $u \ge w$ on $\Omega$ hence
\begin{equation}\label{HD2star}u(x) \geq \delta |z| - N x_n.\end{equation}
We look at the section $S_h$ of $u$. From \eqref{HDstar}-\eqref{HD2star} we see that
\begin{equation}\label{HD3star}S_h \subset \{x_n > \frac{1}{N}(\delta |z| - h)\} \cap \{x_n \leq Mh^{1/2}\}.\end{equation}

We notice that an affine transformation $x \rightarrow Tx,$
$$Tx := x+ \nu_1 z_1+ \nu_2 z_2+ \ldots + \nu_{n-k-1} z_{n-k-1} + \nu_{n-k}x_n$$ with $$\nu_1, \nu_2, \ldots, \nu_{n-k} \in span\{e_1, \ldots, e_k\}$$
i.e  a {\it sliding along the $y$ direction}, leaves the $z, x_n$ coordinate invariant together with the subspace $(y, 0, 0).$

The section $\tilde S_h:=TS_h$ of the rescaling
$$\tilde u(Tx) = u(x)$$ satisfies \eqref{HD3star} and $\tilde u=\tilde \varphi$ on $\p \tilde S_h$ with
$$ \tilde \varphi= \varphi \quad \mbox{on $\tilde G:=\{\varphi \le h\} \subset G$},$$
$$ \tilde \varphi=h \quad \mbox{on $\p \tilde S_h \setminus \tilde G$}.$$

 From John's lemma we know that $S_h$ is equivalent to an ellipsoid $E_h$. We choose $T$ an appropriate sliding along the $y$ direction, so that $TE_h$ becomes symmetric with respect to the $y$ and $(z,x_n)$ subspaces, thus
 $$\tilde x_h^* + c(n) |\tilde S_h|^{1/n} AB_1 \subset \tilde S_h \subset C(n) |\tilde S_h|^{1/n} A B_1, \quad \det A=1$$
 and the matrix $A$ leaves the $y$ and the $(z,x_n)$ subspaces invariant.

 By choosing an appropriate system of coordinates in the $y$ and $z$ variables we may assume
 $$A(y,z,x_n) = (A_1 y, A_2(z,x_n))$$ with
\[
 A_{1} =
 \begin{pmatrix}
  \beta_{1} & 0 & \cdots & 0 \\
  0 & \beta_{2} & \cdots & 0 \\
  \vdots  & \vdots  & \ddots & \vdots  \\
  0 & 0 & \cdots & \beta_{k}
 \end{pmatrix}
\]
with $0<\beta_1 \le \cdots \le \beta_k$, and

\[
 A_{2} =
 \begin{pmatrix}
  \gamma_{k+1} & 0 & \cdots & 0 &\theta_{k+1} \\
  0 & \gamma_{k+2} & \cdots & 0 & \theta_{k+2} \\
  \vdots  & \vdots  & \ddots & \vdots & \vdots  \\
  0 & 0 & \cdots & \gamma_{n-1} & \theta_{n-1}\\
  0 & 0 & \cdots & 0& \theta_{n}
 \end{pmatrix}
\]
with $\gamma_j$, $\theta_n >0$.

 Next we use the induction hypothesis and show that $\tilde S_h$ is equivalent to a ball.

\begin{lem} \label{ball}
There exists $C_0$ such that $$ \tilde S_h \subset C_0 h^{n/2} B_1^+.$$
 \end{lem}

\begin{proof} Using that $$|\tilde S_h| \sim h^{n/2}$$ we obtain
 $$\tilde x_h^* + c h^{1/2} AB_1 \subset \tilde S_h \subset C h^{1/2} AB_1.$$
We need to show that $$\|A\| \le C.$$

Since $\tilde S_h$ satisfies \eqref{HD3star} we see that
$$\tilde S_h \subset \{ |(z,x_n)| \le C h^{1/2}\},$$
which together with the inclusion above gives $\|A_2\| \le C$ hence
$$\gamma_j, \theta_n \leq C, \quad |\theta_j| \leq C.$$
Also $\tilde S_h$ contains the set $$\{(y,0,0) | \quad |y| \leq \tilde \mu^{1/2}h^{1/2}\} \subset \tilde G,$$
which implies $$\beta_i \geq c >0, \quad  i=1,\cdots,k.$$

We define the rescaling
$$w(x) = \frac 1 h \tilde u (h^{1/2} Ax)$$ which is defined in a domain $\Omega_w:=h^{-1/2}A^{-1}\tilde S_h$ such that
$$B_c(x_0) \subset \Omega_w \subset B^+_C, \quad 0\in \p \Omega_w, $$
and $w=\varphi_w$ on $\p \Omega_w$ with
$$\varphi_w=1 \quad \text{on $\p \Omega_w \setminus G_w$},$$

$$\tilde \mu \sum \beta_i^2 x_i^2 \leq \varphi_w \leq \min \{1, \, \tilde \mu^{-1} \sum \beta_i^2 x_i^2\} \ \quad \text{on $G_w$}, $$ where $G_w:=h^{-1/2}A^{-1}\tilde G$.

This implies that $$w \in \mathcal{D}^{\bar \mu}_0(\beta_1, \beta_2, \ldots, \beta_k, \infty, \ldots, \infty)$$ for some value $\bar {\mu}$ depending on $\mu, M, \lambda, \Lambda, n,k$.

We claim that $$b_u(h) \ge c_\star.$$ First we notice that $$b_u(h)=b_{\tilde u}(h) \sim \theta_n.$$

Since
$$\theta_n \prod \beta_i \prod \gamma_j  =\det A=1$$ and $$\gamma_j \leq C,$$ we see that if $b_u(h)$ (and therefore $\theta_n$) becomes smaller than a critical value $c_*$ then
$$\beta_k \geq C_k(\bar \mu, \bar M, \lambda, \Lambda,n),$$
with $\bar M:=2 \bar \mu^{-1}$, and by the induction hypothesis
$$b_w(\tilde h) \geq \bar M \ge 2 b_w(1)$$ for some $\tilde h > C_k^{-1}$. This gives
$$\frac{b_u(h \tilde h)}{b_u(h)}=\frac{b_w(\tilde h)}{b_w(1)} \ge 2,$$
which implies $b_u(h \tilde h) \ge 2 b_u(h)$ and our claim follows.

Next we claim that $\gamma_j$ are bounded below by the same argument. Indeed, from the claim above $\theta_n$ is bounded below and if some $\gamma_j$ is smaller than a small value $\tilde c_*$ then
$$\beta_k \geq C_k(\bar \mu, \bar M_1,\lambda, \Lambda,n)$$ with $$\bar M_1:=\frac{2M}{\bar \mu c_\star}.$$ By the induction hypothesis $$b_w(\tilde h) \geq \bar M_1 \geq \frac{2M}{c_\star} b_w(1),$$ hence
$$\frac{b_u(h \tilde h)}{b_u(h)} \geq \frac{2M}{c_\star}$$ which gives $b_u(h \tilde h) \ge 2M$, contradiction.
In conclusion $\theta_n$, $\gamma_j$ are bounded below which implies that $\beta_i$ are bounded above. This shows that $\|A\|$ is bounded and the lemma is proved.

\end{proof}

Next we use the lemma above and show that the function $u$ has the following property.

\begin{lem}\label{lastlem}If for some $p, q>0$, $$u \geq p(|z| - q x_n), \quad \quad q \le q_0$$ then
$$u \geq p'(|z| - (q-\eta) x_n)$$ for some $p' \ll p,$ and with $\eta>0$ depending on $q_0$ and $\mu, M, \lambda, \Lambda, n,k$.  \end{lem}

\begin{proof} From Lemma \ref{ball} we see that after performing a linear transformation $T$ (siding along the $y$ direction) we may assume that
$$S_h \subset C_0h^{1/2}B_1.$$

Let $$w(x):= \frac{1}{h} u(h^{1/2} x)$$ for some small $h\ll p.$ 

Then $$S_1(w) := \Omega_w=h^{-1/2} S_h \subset B^+_{C_0}$$ and our hypothesis becomes
\begin{equation}\label{w} w \geq \frac{p}{h^{1/2}} (|z| -q x_n),
 \end{equation} Moreover the boundary values $\varphi_w$ of $w$ on $\p \Omega_w$ satisfy
 $$\varphi_w=1 \quad \text{on $\p \Omega_w \setminus G_w$}$$
 $$\tilde \mu |y|^2 \le \varphi_w \le \min\{1,\tilde \mu^{-1}|y|^2\} \quad \mbox{on} \quad G_w,$$
where $G_w:=h^{-1/2}\{\varphi \le h\}$.

Next we show that $\varphi_w \ge v$ on $\p \Omega_w$ where $v$ is defined as
$$v := \delta |x|^2 + \frac{\Lambda}{\delta^{n-1}}(z_1-qx_n)^2 + N(z_1-qx_n) +\delta x_n, $$ and $\delta$ is small depending on $\tilde \mu$ and $C_0$, and $N$ is chosen large such that $$\frac{\Lambda}{\delta^{n-1}} t^2 + Nt$$ is increasing in the interval $|t|\le (1+q_0)C_0.$

From the definition of $v$ we see that $$\det D^2v > \Lambda.$$

On the part of the boundary $\p \Omega_w $ where $z_1 \le qx_n$ we use that $\Omega_w \subset B_{C_0}$ and obtain
$$v \le \delta (|x|^2+x_n) \le \varphi_w.$$

On the part of the boundary $\p \Omega_w $ where $z_1 > qx_n$ we use \eqref{w} and obtain
$$1=\varphi_w \geq C(|z| -q x_n) \ge C(z_1-qx_n)$$ with $C$ arbitrarily large provided that $h$ is small enough. We choose $C$ such that
the inequality above implies $$\frac{\Lambda}{\delta^{n-1}}(z_1-qx_n)^2 + N(z_1 -qx_n) <\frac 1 2.$$ Then
$$\varphi_w=1 > \frac 12 + \delta (|x|^2+x_n) \ge v.$$

In conclusion $\varphi_w \ge v$ on $\p \Omega_w$ hence the function $v$ is a lower barrier for $w$ in $\Omega_w$. Then
$$w \geq N(z_1 -qx_n)+\delta x_n$$ and, since this inequality holds for all directions in the $z$-plane, we obtain  $$w \geq N(|z| -(q-\eta)x_n), \quad \quad \eta := \frac \delta N.$$ Scaling back we get $$u \ge p'(|z| -(q-\eta)x_n) \quad \quad \mbox{in $S_h$}.$$ Since $u$ is convex and $u(0)=0$, this inequality holds globally, and the lemma is proved.

\end{proof}

We remark that Lemma \ref{lastlem} can be used directly to prove Proposition \ref{2dprop} and Lemma \ref{base}.

\

{\it End of the proof of Proposition \ref{HDprop}.}
From \eqref{HD2star} we obtain an initial pair $(p,q_0)$ which satisfies the hypothesis of Lemma \ref{lastlem}. We apply this lemma a finite number of times and obtain that $$u \ge \eps(|z|+x_n),$$ and we contradict that $\tilde S_h$ is equivalent to a ball of radius $h^{1/2}.$

\qed

\end{document}